\documentclass{amsart}
\usepackage{amsmath}
\usepackage{amsfonts}

\newtheorem{theorem}{Theorem}[section]
\newtheorem{corollary}[theorem]{Corollary}

\newtheorem{lemma}[theorem]{Lemma}

\newtheorem{proposition}[theorem]{Proposition}

\newtheorem{example}[theorem]{Example}

\input{tcilatex}

\begin{document}
\title[\textbf{On a class of $\lowercase{h}$-Fourier integral operators}]{%
\textbf{On a class of $\lowercase{h}$-Fourier integral operators}}
\author{\textbf{Harrat Chahrazed}}
\address[Harrat C.]{Universit\'{e} des Sciences et Technlogie Mohamed Boudiaf%
\\
Facult\'{e} des Sciences, D\'{e}partement de Math\'{e}matiques}
\email{chahrakha@yahoo.fr}
\author{\textbf{Senoussaoui Abderrahmane}}
\address[Senoussaoui A.]{Universit\'{e} d'Oran, Facult\'{e} des Sciences\\
D\'{e}partement de\ Math\'{e}matiques. B.P. 1524 El-Mnaouer, Oran,
ALGERIA.} \email{senoussaoui\_abdou@yahoo.fr}

\begin{abstract}
In this paper, we study the $L^{2}$-boundedness and $L^{2}$-compactness of a
class of $h$-Fourier integral operators. These operators are bounded
(respectively compact) if the weight of the amplitude is bounded
(respectively tends to $0)$.
\end{abstract}

\maketitle

\section{Introduction}

For $\varphi \in \mathcal{S}\left( \mathbb{R}^{n}\right) $ (the Schwartz
space), the integral operators%
\begin{equation}
F_{h}\varphi \left( x\right) =\iint e^{\frac{i}{h}\left( S\left( x,\theta
\right) -y\theta \right) }a\left( x,\theta \right) \varphi \left( y\right)
dyd\theta  \label{1.1}
\end{equation}%
appear naturally in the expression of the solutions of the semiclassical
hyperbolic partial differential equations and in the expression of the $%
C^{\infty }$-solution of the associate Cauchy's problem. Which appear two $%
C^{\infty }$-functions, the phase function $\phi \left( x,y,\theta \right)
=S\left( x,\theta \right) -y\theta $ and the amplitude $a.$.

Since 1970, many efforts have been made by several authors in order to study
these type of operators (see, e.g.,\cite{AsFu,Ha,He,Du,Ho1}). The first
works on Fourier integral operators deal with local properties. On the other
hand, K. Asada and D. Fujiwara (\cite{AsFu}) have studied for the first time
a class of Fourier integral operators defined on $\mathbb{R}^{n}.$

For the $h$-Fourier integral operators, an interesting question is under
which conditions on $a$ and $S$ these operators are bounded on $L^{2}$ or
are compact on $L^{2}$.

It has been proved in \cite{AsFu} by a very elaborated proof and with some
hypothesis on the phase function $\phi $ and the amplitude $a$ that all
operators of the form:

\begin{equation}
\left( I\left( a,\phi \right) \varphi \right) \left( x\right) =\underset{%
\mathbb{R}_{y}^{n}\times \mathbb{R}_{\theta }^{N}}{\iint }e^{i\phi \left(
x,\theta ,y\right) }a\left( x,\theta ,y\right) \varphi \left( y\right)
dyd\theta  \label{1.2}
\end{equation}
are bounded on $L^{2}$ where, $x\in \mathbb{R}^{n},$ $n\in \mathbb{N}^{\ast
} $ and $N\in \mathbb{N}$ (if $N=0,$ $\theta $ doesn't appear in $\left( \ref%
{1.2}\right) $). The technique used there is based on the fact that the
operators $I\left( a,\phi \right) I^{\ast }\left( a,\phi \right) ,I^{\ast
}\left( a,\phi \right) I\left( a,\phi \right) $ are pseudodifferential and
it uses Cald\'{e}ron-Vaillancourt's theorem (here $I\left( a,\phi \right)
^{\ast }$ is the adjoint of $I\left( a,\phi \right) )$.

In this work, we apply the same technique of \cite{AsFu} to establish the
boundedness and the compactness of the operators $\left( \ref{1.1}\right) $.
To this end we give a brief and simple proof for a result of \cite{AsFu} in
our framework.

We mainly prove the continuity of the operator $F_{h}$ on $L^{2}\left(
\mathbb{R}^{n}\right) $ when the weight of the amplitude $a$ is bounded.
Moreover, $F_{h}$ is compact on $L^{2}\left( \mathbb{R}^{n}\right) $ if this
weight tends to zero. Using the estimate given in \cite{Ro,Se} for $h$%
-pseudodifferential ($h$-admissible) operators, we also establish an $L^{2}$%
-estimate of $\left\Vert F_{h}\right\Vert .$

We note that if the amplitude $a$ is juste bounded, the Fourier integral
operator $F$ is not necessarily bounded on $L^{2}\left( \mathbb{R}%
^{n}\right) .$ Recently, M. Hasanov \cite{Ha} and we \cite{AIS} constructed
a class of unbounded Fourier integral operators with an amplitude in the H%
\"{o}rmander's class $S_{1,1}^{0}$ and in $\bigcap\limits_{0<\rho <1}S_{\rho
,1}^{0}.$

To our knowledge, this work constitutes a first attempt to diagonalize the $h
$-Fourier integral operators on $L^{2}\left( \mathbb{R}^{n}\right) $
(relying on the compactness of these operators).


\section{A general class of $h$-Fourier integral operators}


If $\varphi \in \mathcal{S}(\mathbb{R}^n)$, we consider the following
integral transformations
\begin{equation}
(I(a,\phi;h )\varphi )(x)=\underset{ \mathbb{R}_{y}^n\times \mathbb{R}%
_{\theta }^{N}}{\iint }e^{\frac{i}{h}\phi (x,\theta ,y)}a(x,\theta
,y)\varphi (y)dy\,d\theta  \label{2.1}
\end{equation}
where, $x\in \mathbb{R}^n$, $n\in \mathbb{N}^{\ast }$ and $N\in \mathbb{N}$
(if $N=0$, $\theta $ doesn't appear in \eqref{2.1}).

In general the integral \eqref{2.1} is not absolutely convergent, so we use
the technique of the oscillatory integral developed by H\"{o}rmander. The
phase function $\phi $ and the amplitude $a$ are assumed to satisfy the
following hypothesis:

\begin{itemize}
\item[(H1)] $\phi \in C^{\infty }(\mathbb{R}_{x}^n\times \mathbb{R} _{\theta
}^{N}\times \mathbb{R}_{y}^n,\mathbb{R})$ ($\phi$ is a real function)

\item[(H2)] For all $(\alpha ,\beta ,\gamma )\in \mathbb{N} ^n\times \mathbb{%
N}^{N}\times \mathbb{N}^n$, there exists $C_{\alpha ,\beta,\gamma }>0$
\begin{equation*}
|\partial _{y}^{\gamma }\partial _{\theta }^{\beta }\partial _{x}^{\alpha
}\phi (x,\theta ,y)| \leq C_{\alpha ,\beta ,\gamma }\lambda ^{(2-|\alpha |
-|\beta | -|\gamma | )_{+}}(x,\theta ,y)
\end{equation*}
where $\lambda (x,\theta ,y)=(1+|x| ^2+|\theta |^2+|y| ^2)^{1/2}$ called the
weight and
\begin{equation*}
(2-|\alpha | -|\beta| -|\gamma | )_{+} =\max (2-|\alpha | -|\beta | -|\gamma
| ,0)
\end{equation*}
\end{itemize}

\begin{itemize}
\item[(H3)] There exist $K_{1},K_{2}>0$ such that $\forall (x,\theta ,y)\in
\mathbb{R}_{x}^n\times \mathbb{R}_{\theta }^{N}\times \mathbb{R}_{y}^n$
\begin{equation*}
K_{1}\lambda (x,\theta ,y)\leq \lambda (\partial _{y}\phi ,\partial _{\theta
}\phi ,y)\leq K_{2}\lambda (x,\theta ,y)
\end{equation*}

\item[(H3*)] There exist $K_{1}^{\ast },K_{2}^{\ast }>0$ such that, $\forall
(x,\theta ,y)\in \mathbb{R}_{x}^n\times \mathbb{R}_{\theta }^{N}\times
\mathbb{R}_{y}^n$
\begin{equation*}
K_{1}^{\ast }\lambda (x,\theta ,y)\leq \lambda (x,\partial _{\theta }\phi
,\partial _{x}\phi )\leq K_{2}^{\ast }\lambda (x,\theta ,y)
\end{equation*}
\end{itemize}

For any open subset $\Omega $ of $\mathbb{R}_{x}^n\times \mathbb{R}_{\theta
}^{N}\times \mathbb{R}_{y}^n$, $\mu \in \mathbb{R}$ and $\rho \in [0,1] $,
we set
\begin{align*}
\Gamma _{\rho }^{\mu }(\Omega ) &=\big\{a\in C^{\infty }(\Omega ): \forall
(\alpha ,\beta ,\gamma )\in \mathbb{N}^n\times \mathbb{N}^{N}\times \mathbb{N%
}^n, \; \exists C_{\alpha ,\beta ,\gamma }>0: \\
&\quad |\partial _{y}^{\gamma }\partial _{\theta }^{\beta }\partial
_{x}^{\alpha }a(x,\theta ,y)| \leq C_{\alpha ,\beta ,\gamma }\lambda ^{\mu
-\rho (|\alpha | +|\beta | +|\gamma| )}(x,\theta ,y)\big\}
\end{align*}

When $\Omega =\mathbb{R}_{x}^n\times \mathbb{R}_{\theta }^{N}\times \mathbb{R%
}_{y}^n$, we denote $\Gamma _{\rho }^{\mu}(\Omega )=\Gamma _{\rho }^{\mu }$.

To give a meaning to the right hand side of \eqref{2.1}, we consider $g\in
\mathcal{S}(\mathbb{R}_{x}^n\times \mathbb{R}_{\theta }^{N}\times \mathbb{R}%
_{y}^n)$, $g(0)=1$. If $a\in \Gamma _{0}^{\mu }$, we define
\begin{equation*}
a_{\sigma }(x,\theta ,y)=g(x/\sigma ,\theta /\sigma ,y/\sigma )a(x,\theta
,y),\quad \sigma >0.
\end{equation*}

\begin{theorem}
If $\phi $ satisfies (H1), (H2), (H3) and (H3*), and if $a\in \Gamma
_{0}^{\mu }$, then

1. For all $\varphi \in \mathcal{S}(\mathbb{R}^n)$, $\lim_{\sigma \to
+\infty } [ I(a_{\sigma },\phi;h)\varphi] (x)$ exists for every point $x\in
\mathbb{R}^n$ and is independent of the choice of the function $g$. We
define
\begin{equation*}
(I(a,\phi;h )\varphi )(x):=\lim_{\sigma \to +\infty} (I(a_{\sigma },\phi;h
)\varphi)(x)
\end{equation*}

2. $I(a,\phi ;h)\in \mathcal{L}(\mathcal{S}(\mathbb{R}^{n}))$ and $I(a,\phi
;h)\in \mathcal{L}(\mathcal{S}^{\prime }(\mathbb{R}^{n}))$ (here $\mathcal{L}%
(E)$ is the space of bounded linear mapping from $E$ to $E$ and $\mathcal{S}%
^{\prime }(\mathbb{R}^{n})$ the space of all distributions with temperate
growth on $\mathbb{R}^{n}$).
\end{theorem}

\begin{proof}
see \cite{He} or \cite[propostion II.2]{Ro}.
\end{proof}

\begin{example}
Let's give two examples of operators of the form $\left( \ref{2.1}\right) $
which satisfy $\left( H1\right) $ to $(H3)^{\ast }$:

\begin{enumerate}
\item The Fourier transform $\mathcal{F}\psi \left( x\right) =$ $\underset{%
\mathbb{R}^{n}}{\int }e^{-ixy}\psi \left( y\right) dy,$ $\psi \in \mathcal{S}%
(\mathbb{R}^{n}),$

\item Pseudodifferential operators $A\psi \left( x\right) =$ $\left( 2\pi
\right) ^{-n}\underset{\mathbb{R}^{2n}}{\int }e^{i\left( x-y\right) \theta
}a\left( x,y,\theta \right) \psi \left( y\right) dyd\theta ,$ $\psi \in
\mathcal{S}(\mathbb{R}^{n}),$ $a\in \Gamma _{0}^{\mu }\left( \mathbb{R}%
^{3n}\right) .$
\end{enumerate}
\end{example}


\section{Assumptions and Preliminaries}


We consider the special form of the phase function
\begin{equation}
\phi (x,y,\theta )=S(x,\theta )-y\theta  \label{3.1}
\end{equation}
where $S$ satisfies

\begin{itemize}
\item[(G1)] $S\in C^{\infty }(\mathbb{R}_{x}^n\times \mathbb{R}_{\theta }^n,%
\mathbb{R})$,

\item[(G2)] For each $(\alpha ,\beta )\in \mathbb{N}^{2n}$, there exist $%
C_{\alpha ,\beta }>0$, such that
\begin{equation*}
|\partial _{x}^{\alpha }\partial_{\theta }^{\beta }S(x,\theta )| \leq
C_{\alpha ,\beta}\lambda (x,\theta )^{(2-|\alpha | -|\beta | )},
\end{equation*}

\item[(G3)] There exists $\delta _{0}>0$ such that
\begin{equation*}
\inf_{x,\theta \in \mathbb{R}^n}|\det \frac{\partial ^2S}{ \partial
x\partial \theta }(x,\theta )| \geq \delta _{0}.
\end{equation*}
\end{itemize}

\begin{lemma}[\protect\cite{MeSe}]
Let's assume that $S$ satisfies (G1), (G2), (G3). Then the function $\phi
(x,y,\theta )=S(x,\theta )-y\theta $ satisfies (H1), (H2), (H3) and (H3*).
\end{lemma}

\begin{lemma}[\protect\cite{MeSe}]
If $S$ satisfies (G1), (G2) and (G3), then there exists $C_{2}>0$ such that
for all $(x,\theta ),(x^{\prime },\theta ^{\prime })\in \mathbb{R}^{2n}$,
\begin{equation}
|x-x^{\prime }|+|\theta -\theta ^{\prime }|\leq C_{2}\big[|(\partial
_{\theta }S)(x,\theta )-(\partial _{\theta }S)(x^{\prime },\theta ^{\prime
})|+|\theta -\theta ^{\prime }|\big]  \label{3.2}
\end{equation}
\end{lemma}

when $\theta =\theta ^{\prime }$ in \eqref{3.2}, there exists $C_{2}>0$,
such that for all $(x,x^{\prime },\theta )\in \mathbb{R}^{3n}$,
\begin{equation}
|x-x^{\prime }| \leq C_{2}|(\partial _{\theta }S)(x,\theta )-(\partial
_{\theta }S)(x^{\prime },\theta )| .  \label{3.3}
\end{equation}

\begin{proposition}
If $S$ satisfies (G1) and (G2), then there exists a constant $\epsilon
_{0}>0 $ such that the phase function $\phi $ given in \eqref{3.1} belongs
to $\Gamma _{1}^{2}(\Omega _{\phi ,\epsilon _{0}})$ where
\begin{equation*}
\Omega _{\phi ,\epsilon _{0}}=\big\{(x,\theta ,y)\in \mathbb{R}%
^{3n};\;\;|\partial _{\theta }S(x,\theta )-y|^{2}<\epsilon
_{0}\;(|x|^{2}+|y|^{2}+|\theta |^{2})\big\}.
\end{equation*}
\end{proposition}

\begin{proof}
We have to show that: $\exists \varepsilon _{0}>0,$ $\forall \alpha ,\beta
,\gamma \in \mathbb{N}^{n}\bigskip ,\exists C_{\alpha ,\beta ,\gamma }>0;$%
\begin{equation}
\left\vert \partial _{x}^{\alpha }\partial _{\theta }^{\beta }\partial
_{y}^{\gamma }\phi (x,\theta ,y)\right\vert \leq C_{\alpha ,\beta ,\gamma
}\lambda (x,\theta ,y)^{(2-\left\vert \alpha \right\vert -\left\vert \beta
\right\vert -\left\vert \gamma \right\vert )},\text{ }\forall (x,\theta
,y)\in \Omega _{\phi ,\varepsilon _{0}}.  \label{3.4}
\end{equation}

\begin{itemize}
\item If $\left\vert \gamma \right\vert =1,$ then $\left\vert \partial
_{x}^{\alpha }\partial _{\theta }^{\beta }\partial _{y}^{\gamma }\phi
(x,\theta ,y)\right\vert =\left\vert \partial _{x}^{\alpha }\partial
_{\theta }^{\beta }\left( -\theta \right) \right\vert =\left\{
\begin{array}{c}
0\;\;\;\;\;\;\;\;\;\;\;\;\text{if }\left\vert \alpha \right\vert \neq 0 \\
\left\vert \partial _{\theta }^{\beta }(-\theta )\right\vert \;\;\text{if }%
\alpha =0%
\end{array}%
\right. ;$

\item If $\left\vert \gamma \right\vert >1,$ then $\left\vert \partial
_{x}^{\alpha }\partial _{\theta }^{\beta }\partial _{y}^{\gamma }\phi
(x,\theta ,y)\right\vert =0.$
\end{itemize}

Hence the estimate $\left( \ref{3.4}\right) $ is satisfied.

If $\left\vert \gamma \right\vert =0,$ then $\forall \alpha ,\beta \in
\mathbb{N}^{n}\bigskip ;$ $\left\vert \alpha \right\vert +\left\vert \beta
\right\vert \leq 2,\ \exists C_{\alpha ,\beta }>0$;%
\begin{equation*}
\left\vert \partial _{x}^{\alpha }\partial _{\theta }^{\beta }\phi (x,\theta
,y)\right\vert =\left\vert \partial _{x}^{\alpha }\partial _{\theta }^{\beta
}S\left( x,\theta \right) -\partial _{x}^{\alpha }\partial _{\theta }^{\beta
}\left( y\theta \right) \right\vert \leq C_{\alpha ,\beta }\lambda (x,\theta
,y)^{(2-\left\vert \alpha \right\vert -\left\vert \beta \right\vert )}%
\newline
.
\end{equation*}%
If $\left\vert \alpha \right\vert +\left\vert \beta \right\vert >2,$ one has
$\partial _{x}^{\alpha }\partial _{\theta }^{\beta }\phi (x,\theta
,y)=\partial _{x}^{\alpha }\partial _{\theta }^{\beta }S\left( x,\theta
\right) .$ In $\Omega _{\phi ,\varepsilon _{0}}$ we have
\begin{equation*}
\left\vert y\right\vert =\left\vert \partial _{\theta }S\left( x,\theta
\right) -y-\partial _{\theta }S\left( x,\theta \right) \right\vert \leq
\sqrt{\varepsilon _{0}}\left( \left\vert x\right\vert ^{2}+\left\vert
y\right\vert ^{2}+\left\vert \theta \right\vert ^{2}\right)
^{1/2}+C_{3}\lambda \left( x,\theta \right) ,\text{ }C_{3}>0.
\end{equation*}%
For $\varepsilon _{0}$ sufficiently small, we obtain a constant $C_{4}>0$
such that%
\begin{equation}
\left\vert y\right\vert \leq C_{4}\lambda \left( x,\theta \right) ,\text{ }%
\forall (x,\theta ,y)\in \Omega _{\phi ,\varepsilon _{0}\;}.  \label{3.5}
\end{equation}%
This inequality leads to the equivalence
\begin{equation}
\lambda \left( x,\theta ,y\right) \simeq \lambda \left( x,\theta \right)
\text{ in }\Omega _{\phi ,\varepsilon _{0}\;}  \label{3.6}
\end{equation}%
thus the assumption $\left( G2\right) $ and $\left( \ref{3.6}\right) $ give
the estimate $\left( \ref{3.4}\right) $.
\end{proof}

Using \eqref{3.6}, we have the following result.

\begin{proposition}
If $(x,\theta )\to a(x,\theta )$ belongs to $\Gamma _{k}^{m}(\mathbb{R}%
_{x}^n\times \mathbb{R}_{\theta }^n)$, then $(x,\theta ,y)\to a(x,\theta )$
belongs to $\Gamma _{k}^{m}(\mathbb{R}_{x}^n\times \mathbb{R}_{\theta
}^n\times \mathbb{R}_{y}^n)\cap \Gamma _{k}^{m}(\Omega _{\phi ,\epsilon
_{0}\;})$, $k\in \{ 0,1\} $.
\end{proposition}


\section{$L^2$-boundedness and $L^2$-compactness of $F_{h}$}


\begin{theorem}
\label{main result}Let $F_{h}$ be the integral operator of distribution
kernel
\begin{equation}
K(x,y;h)=\int_{\mathbb{R}^{n}}e^{\frac{i}{h}(S(x,\theta )-y\theta
)}a(x,\theta )\widehat{d_{h}\theta }  \label{4.1}
\end{equation}%
where $\widehat{d_{h}\theta }=(2\pi h)^{-n}\,d\theta $, $a\in \Gamma
_{k}^{m}(\mathbb{R}_{x,\theta }^{\;2n})$, $k=0,1$ and $S$ satisfies $(G1)$,
(G2) and (G3). Then $F_{h}F_{h}^{\ast }$ and $F_{h}^{\ast }F_{h}$ are $h$%
-pseudodifferential operators with symbol in $\Gamma _{k}^{2m}(\mathbb{R}%
^{2n})$, $k=0,1$, given by
\begin{gather*}
\sigma (F_{h}F_{h}^{\ast })(x,\partial _{x}S(x,\theta ))\equiv |a(x,\theta
)|^{2}|(\det \frac{\partial ^{2}S}{\partial \theta \partial x}%
)^{-1}(x,\theta )| \\
\sigma (F_{h}^{\ast }F_{h})(\partial _{\theta }S(x,\theta ),\theta )\equiv
|a(x,\theta )|^{2}|(\det \frac{\partial ^{2}S}{\partial \theta \partial x}%
)^{-1}(x,\theta )|
\end{gather*}%
we denote here $a\equiv b$ for $a,b\in \Gamma _{k}^{2p}(\mathbb{R}^{2n})$ if
$(a-b)\in \Gamma _{k}^{2p-2}(\mathbb{R}^{2n})$ and $\sigma $ stands for the
symbol.
\end{theorem}

\begin{proof}
For all $v\in \mathcal{S}(\mathbb{R}^{n})$, we have:
\begin{equation}
(F_{h}F_{h}^{\ast }v)(x)=\int_{\mathbb{R}^{n}}\int_{\mathbb{R}^{n}}e^{\frac{i%
}{h}(S(x,\theta )-S(\widetilde{x},\theta ))}a(x,\theta )\overline{a}(%
\widetilde{x},\theta )d\widetilde{x}\widehat{d\theta }.  \label{4.2}
\end{equation}%
The main idea to show that $F_{h}F_{h}^{\ast }$ is a $h$- pseudodifferential
operator, is to use the fact that $(S(x,\theta )-S(\widetilde{x},\theta ))$
can be expressed by the scalar product $\langle x-\widetilde{x},\xi (x,%
\widetilde{x},\theta )\rangle $ after considering the change of variables $%
(x,\widetilde{x},\theta )\rightarrow (x,\widetilde{x},\xi =\xi (x,\widetilde{%
x},\theta ))$. The distribution kernel of $F_{h}F_{h}^{\ast }$ is
\begin{equation*}
K(x,\tilde{x};h)=\int_{\mathbb{R}^{n}}e^{\frac{i}{h}(S(x,\theta )-S(\tilde{x}%
,\theta ))}a(x,\theta )\overline{a}(\tilde{x},\theta )\widehat{d_{h}\theta }.
\end{equation*}%
We obtain from $\left( \ref{3.3}\right) $ that if%
\begin{equation*}
\left\vert x-\widetilde{x}\right\vert \geq \frac{\varepsilon }{2}\lambda
\left( x,\widetilde{x},\theta \right) \text{ (where }\varepsilon >0\text{ is
sufficiently small})
\end{equation*}%
then
\begin{equation}
\left\vert \left( \partial _{\theta }S\right) (x,\theta )-\left( \partial
_{\theta }S\right) \left( \widetilde{x},\theta \right) \right\vert \geq
\frac{\varepsilon }{2C_{2}}\lambda \left( x,\widetilde{x},\theta \right) .
\label{4.5}
\end{equation}%
Choosing $\omega \in C^{\infty }(\mathbb{R})$ such that
\begin{gather*}
\omega (x)\geq 0,\quad \forall x\in \mathbb{R} \\
\omega (x)=1\quad \text{if }x\in \lbrack -\frac{1}{2},\frac{1}{2}] \\
\mathop{\rm supp}\omega \subset ]-1,1[
\end{gather*}%
and setting
\begin{gather*}
b(x,\tilde{x},\theta ):=a(x,\theta )\overline{a}(\tilde{x},\theta
)=b_{1,\epsilon }(x,\tilde{x},\theta )+b_{2,\epsilon }(x,\tilde{x},\theta )
\\
b_{1,\epsilon }(x,\tilde{x},\theta )=\omega (\frac{|x-\tilde{x}|}{\epsilon
\lambda (x,\tilde{x},\theta )})b(x,\tilde{x},\theta ) \\
b_{2,\epsilon }(x,\tilde{x},\theta )=[1-\omega (\frac{|x-\tilde{x}|}{%
\epsilon \lambda (x,\tilde{x},\theta )})]b(x,\tilde{x},\theta ).
\end{gather*}%
We have $K(x,\widetilde{x};h)=K_{1,\epsilon }(x,\widetilde{x}%
;h)+K_{2,\epsilon }(x,\widetilde{x};h)$, where
\begin{equation*}
K_{j,\epsilon }(x,\tilde{x};h)=\int_{\mathbb{R}^{n}}e^{\frac{i}{h}%
(S(x,\theta )-S(\tilde{x},\theta ))}b_{j,\epsilon }(x,\tilde{x},\theta )%
\widehat{d_{h}\theta },\quad j=1,2.
\end{equation*}%
We will study separately the kernels $K_{1,\epsilon }$ and $K_{2,\epsilon }$.
\end{proof}

\begin{proof}
For all $h$, we have
\begin{equation*}
K_{2,\epsilon }(x,\widetilde{x};h)\in \mathcal{S}(\mathbb{R}^{n}\times
\mathbb{R}^{n}).
\end{equation*}%
Indeed, using the oscillatory integral method, there is a linear partial
differential operator $L$ of order 1 such that%
\begin{equation*}
L\left( e^{\frac{i}{h}\left( S(x,\theta )-S(\tilde{x},\theta )\right)
}\right) =e^{\frac{i}{h}\left( S(x,\theta )-S(\tilde{x},\theta )\right) }%
\text{ }
\end{equation*}%
\begin{equation*}
\text{where }L=-ih\left\vert \left( \partial _{\theta }S\right) (x,\theta
)-\left( \partial _{\theta }S\right) \left( \widetilde{x},\theta \right)
\right\vert ^{-2}\sum\limits_{l=1}^{n}\left[ \left( \partial _{\theta
_{l}}S\right) (x,\theta )-\left( \partial _{\theta _{l}}S\right) \left(
\widetilde{x},\theta \right) \right] \partial _{\theta _{l}}.
\end{equation*}%
The transpose operator of $L$ is
\begin{equation*}
^{t}L=\sum\limits_{l=1}^{n}F_{l}\left( x,\widetilde{x},\theta ;h\right)
\partial _{\theta _{l}}+G\left( x,\widetilde{x},\theta ;h\right) \text{ }
\end{equation*}%
where $F_{l}\left( x,\widetilde{x},\theta \right) \in \Gamma _{0}^{-1}\left(
\Omega _{\varepsilon }\right) $, $G\left( x,\widetilde{x},\theta \right) \in
\Gamma _{0}^{-2}\left( \Omega _{\varepsilon }\right) $
\begin{equation*}
\left\{
\begin{array}{l}
F_{l}\left( x,\widetilde{x},\theta ;h\right) =ih\left\vert \left( \partial
_{\theta }S\right) (x,\theta )-\left( \partial _{\theta }S\right) \left(
\widetilde{x},\theta \right) \right\vert ^{-2}\left( \left( \partial
_{\theta _{l}}S\right) (x,\theta )-\left( \partial _{\theta _{l}}S\right)
\left( \widetilde{x},\theta \right) \right) \ \ \ \ \ \ \ \ \ \ \  \\
G\left( x,\widetilde{x},\theta ;h\right) =ih\sum\limits_{l=1}^{n}\partial
_{\theta _{l}}\left[ \left\vert \left( \partial _{\theta }S\right) (x,\theta
)-\left( \partial _{\theta }S\right) \left( \widetilde{x},\theta \right)
\right\vert ^{-2}\left( \left( \partial _{\theta _{l}}S\right) (x,\theta
)-\left( \partial _{\theta _{l}}S\right) \left( \widetilde{x},\theta \right)
\right) \right] \\
\Omega _{\varepsilon }=\left\{ \left( x,\tilde{x},\theta \right) \in \mathbb{%
R}^{3n};\;\left\vert \partial _{\theta }S(x,\theta )-\partial _{\theta
}S\left( \tilde{x},\theta \right) \right\vert >\frac{\varepsilon }{2C_{2}}%
\lambda \left( x,\tilde{x},\theta \right) \newline
\right\} .\;\;\;\;\;\;\;\;\;\;\;\;\;\;\;\;\;\ \ \ \ \ \;\;\;\;\;\;%
\end{array}%
\right.
\end{equation*}%
On the other hand we prove by induction on $q$ that
\begin{equation*}
\left( ^{t}L\right) ^{q}b_{2,\varepsilon }\left( x,\tilde{x},\theta \right)
=\sum\limits_{\substack{ \left\vert \gamma \right\vert \leq q  \\ \gamma \in
\mathbb{N}^{n}}}g_{\gamma ,q}\left( x,\tilde{x},\theta \right) \partial
_{\theta }^{\gamma }b_{2,\varepsilon }\left( x,\tilde{x},\theta \right) ,%
\text{ }g_{\gamma }^{\left( q\right) }\in \Gamma _{0}^{-q}\left( \Omega
_{\varepsilon }\right) ,
\end{equation*}%
and so,%
\begin{equation*}
K_{2,\varepsilon }\left( x,\tilde{x}\right) =\int\limits_{\mathbb{R}^{n}}e^{%
\frac{i}{h}\left( S(x,\theta )-S(\tilde{x},\theta )\right) }\left(
^{t}L\right) ^{q}b_{2,\varepsilon }\left( x,\tilde{x},\theta \right)
\widehat{d\theta }.
\end{equation*}

Using Leibnitz's formula, $\left( G2\right) $ and the form $\left(
^{t}L\right) ^{q},$ we can choose $q$ large enough such that%
\begin{equation*}
\forall \alpha ,\alpha ^{\prime },\beta ,\beta ^{\prime }\in \mathbb{N}%
^{n},\exists C_{\alpha ,\alpha ^{\prime },\beta ,\beta ^{\prime }}>0,\text{ }%
\sup_{x,\widetilde{x}\in \mathbb{R}^{n}}\left\vert x^{\alpha }\widetilde{x}%
^{\alpha ^{\prime }}\partial _{x}^{\beta }\partial _{\widetilde{x}}^{\beta
^{\prime }}K_{2,\varepsilon }\left( x,\widetilde{x};h\right) \right\vert
\leq C_{\alpha ,\alpha ^{\prime },\beta ,\beta ^{\prime }}.
\end{equation*}

Next, we study $K_{1}^{\epsilon }$: this is more difficult and depends on
the choice of the parameter $\epsilon $. It follows from Taylor's formula
that
\begin{gather*}
S(x,\theta )-S(\widetilde{x},\theta )=\langle x-\widetilde{x},\xi (x,%
\widetilde{x},\theta )\rangle _{\mathbb{R}^{n}}, \\
\xi (x,\widetilde{x},\theta )=\int_{0}^{1}(\partial _{x}S)(\widetilde{x}+t(x-%
\widetilde{x}),\theta )dt.
\end{gather*}%
We define the vectorial function
\begin{equation*}
\widetilde{\xi }_{\epsilon }(x,\widetilde{x},\theta )=\omega \big(\frac{|x-%
\tilde{x}|}{2\epsilon \lambda (x,\tilde{x},\theta )}\big)\xi (x,\widetilde{x}%
,\theta )+\big(1-\omega (\frac{|x-\tilde{x}|}{2\epsilon \lambda (x,\tilde{x}%
,\theta )})\big)(\partial _{x}S)(\widetilde{x},\theta ).
\end{equation*}%
We have
\begin{equation*}
\widetilde{\xi }_{\varepsilon }\left( x,\widetilde{x},\theta \right) =\xi
\left( x,\widetilde{x},\theta \right) \text{ on }\limfunc{supp}%
b_{1,\varepsilon }.
\end{equation*}%
Moreover, for $\varepsilon $ sufficiently small,%
\begin{equation}
\lambda \left( x,\theta \right) \simeq \lambda \left( \widetilde{x},\theta
\right) \simeq \lambda \left( x,\widetilde{x},\theta \right) \text{ on }%
\limfunc{supp}b_{1,\varepsilon }.  \label{4.6}
\end{equation}%
Let us consider the mapping
\begin{equation}
\mathbb{R}^{3n}\ni \left( x,\widetilde{x},\theta \right) \rightarrow \left(
x,\widetilde{x},\widetilde{\xi }_{\varepsilon }\left( x,\widetilde{x},\theta
\right) \right)  \label{4.7}
\end{equation}%
for which Jacobian matrix is%
\begin{equation*}
\left(
\begin{array}{lll}
I_{n} & 0 & 0 \\
0 & I_{n} & 0 \\
\partial _{x}\widetilde{\xi }_{\varepsilon } & \partial _{\widetilde{x}}%
\widetilde{\xi }_{\varepsilon } & \partial _{\theta }\widetilde{\xi }%
_{\varepsilon }%
\end{array}%
\right) .
\end{equation*}%
We have%
\begin{gather*}
\frac{\partial \widetilde{\xi }_{\varepsilon ,j}}{\partial \theta _{i}}%
\left( x,\widetilde{x},\theta \right) =\frac{\partial ^{2}S}{\partial \theta
_{i}\partial x_{j}}\left( \widetilde{x},\theta \right) +\omega \left( \frac{%
\left\vert x-\tilde{x}\right\vert }{2\varepsilon \lambda \left( x,\tilde{x}%
,\theta \right) }\right) \left( \frac{\partial \xi _{j}}{\partial \theta _{i}%
}\left( x,\widetilde{x},\theta \right) -\frac{\partial ^{2}S}{\partial
\theta _{i}\partial x_{j}}\left( \widetilde{x},\theta \right) \right) \\
-\frac{\left\vert x-\tilde{x}\right\vert }{2\varepsilon \lambda \left( x,%
\tilde{x},\theta \right) }\frac{\partial \lambda }{\partial \theta _{i}}%
\left( x,\tilde{x},\theta \right) \lambda ^{-1}\left( x,\tilde{x},\theta
\right) \omega ^{\prime }\left( \frac{\left\vert x-\tilde{x}\right\vert }{%
2\varepsilon \lambda \left( x,\tilde{x},\theta \right) }\right) \left( \xi
_{j}\left( x,\widetilde{x},\theta \right) -\frac{\partial S}{\partial x_{j}}%
\left( \widetilde{x},\theta \right) \right) .
\end{gather*}%
Thus, we obtain%
\begin{gather*}
\left\vert \frac{\partial \widetilde{\xi }_{\varepsilon ,j}}{\partial \theta
_{i}}\left( x,\widetilde{x},\theta \right) -\frac{\partial ^{2}S}{\partial
\theta _{i}\partial x_{j}}\left( \widetilde{x},\theta \right) \right\vert
\leq \left\vert \omega \left( \frac{\left\vert x-\tilde{x}\right\vert }{%
2\varepsilon \lambda \left( x,\tilde{x},\theta \right) }\right) \right\vert
\left\vert \frac{\partial \xi _{j}}{\partial \theta _{i}}\left( x,\widetilde{%
x},\theta \right) -\frac{\partial ^{2}S}{\partial \theta _{i}\partial x_{j}}%
\left( \widetilde{x},\theta \right) \right\vert + \\
\lambda ^{-1}\left( x,\tilde{x},\theta \right) \left\vert \omega ^{\prime
}\left( \frac{\left\vert x-\tilde{x}\right\vert }{2\varepsilon \lambda
\left( x,\tilde{x},\theta \right) }\right) \right\vert \left\vert \xi
_{j}\left( x,\widetilde{x},\theta \right) -\frac{\partial S}{\partial x_{j}}%
\left( \widetilde{x},\theta \right) \right\vert .
\end{gather*}%
Now it follows from $\left( G2\right) ,$ $\left( \ref{4.6}\right) $ and
Taylor's formula that
\begin{eqnarray}
\left\vert \frac{\partial \xi _{j}}{\partial \theta _{i}}\left( x,\widetilde{%
x},\theta \right) -\frac{\partial ^{2}S}{\partial \theta _{i}\partial x_{j}}%
\left( \widetilde{x},\theta \right) \right\vert &\leq
&\int\limits_{0}^{1}\left\vert \frac{\partial ^{2}S}{\partial \theta
_{i}\partial x_{j}}\left( \widetilde{x}+t\left( x-\widetilde{x}\right)
,\theta \right) -\frac{\partial ^{2}S}{\partial \theta _{i}\partial x_{j}}%
\left( \widetilde{x},\theta \right) \right\vert dt  \notag \\
&\leq &C_{5}\left\vert x-\widetilde{x}\right\vert \lambda ^{-1}\left( x,%
\tilde{x},\theta \right) ,\text{ }C_{5}>0  \label{4.8}
\end{eqnarray}%
\begin{eqnarray}
\left\vert \xi _{j}\left( x,\widetilde{x},\theta \right) -\frac{\partial S}{%
\partial x_{j}}\left( \widetilde{x},\theta \right) \right\vert &\leq
&\int\limits_{0}^{1}\left\vert \frac{\partial S}{\partial x_{j}}\left(
\widetilde{x}+t\left( x-\widetilde{x}\right) ,\theta \right) -\frac{\partial
S}{\partial x_{j}}\left( \widetilde{x},\theta \right) \right\vert dt  \notag
\\
&\leq &C_{6}\left\vert x-\widetilde{x}\right\vert ,\text{ }C_{6}>0\text{ .}
\label{4.9}
\end{eqnarray}%
From $\left( \ref{4.8}\right) $ and $\left( \ref{4.9}\right) ,$ there exists
a positive constant $C_{7}>0,$ such that
\begin{equation}
\left\vert \frac{\partial \widetilde{\xi }_{\varepsilon ,j}}{\partial \theta
_{i}}\left( x,\widetilde{x},\theta \right) -\frac{\partial ^{2}S}{\partial
\theta _{i}\partial x_{j}}\left( \widetilde{x},\theta \right) \right\vert
\leq C_{7}\varepsilon ,\text{ }\forall i,j\in \left\{ 1,...,n\right\} .
\label{4.10}
\end{equation}%
If $\varepsilon <\frac{\delta _{0}}{2\widetilde{C}},$ then $\left( \ref{4.10}%
\right) $ and $\left( G3\right) $ yields the estimate
\begin{equation}
\delta _{0}/2\leq -\widetilde{C}\varepsilon +\delta _{0}\leq -\widetilde{C}%
\varepsilon +\det \frac{\partial ^{2}S}{\partial x\partial \theta }(x,\theta
)\leq \det \partial _{\theta }\widetilde{\xi }_{\varepsilon }\left( x,%
\widetilde{x},\theta \right) ,\text{ with }\widetilde{C}>0.  \label{4.11}
\end{equation}%
If $\varepsilon $ is such that $\left( \ref{4.6}\right) $ and $\left( \ref%
{4.11}\right) $ are true, then the mapping given in $\left( \ref{4.7}\right)
$ is a global diffeomorphism of $\mathbb{R}^{3n}.$ Hence there exists a
mapping%
\begin{equation*}
\theta :\mathbb{R}^{n}\times \mathbb{R}^{n}\times \mathbb{R}^{n}\ni \left( x,%
\widetilde{x},\xi \right) \rightarrow \theta \left( x,\widetilde{x},\xi
\right) \in \mathbb{R}^{n}
\end{equation*}%
such that%
\begin{equation}
\left\{
\begin{array}{ccc}
\widetilde{\xi }_{\varepsilon }\left( x,\widetilde{x},\theta \left( x,%
\widetilde{x},\xi \right) \right) = & \xi &  \\
\theta \left( x,\widetilde{x},\widetilde{\xi }_{\varepsilon }\left( x,%
\widetilde{x},\theta \right) \right) = & x &  \\
\partial ^{\alpha }\theta \left( x,\widetilde{x},\xi \right) =\mathcal{O}%
\left( 1\right) , &  & \forall \alpha \in \mathbb{N}^{3n}\backslash \left\{
0\right\}%
\end{array}%
\right.  \label{4.12}
\end{equation}%
If we change the variable $\xi $ by $\theta \left( x,\widetilde{x},\xi
\right) $ in $K_{1,\varepsilon }\left( x,\widetilde{x}\right) $, we obtain:%
\begin{equation}
K_{1,\varepsilon }\left( x,\widetilde{x}\right) =\int\limits_{\mathbb{R}%
^{n}}e^{i<x-\tilde{x},\xi >}b_{1,\varepsilon }\left( x,\tilde{x},\theta
\left( x,\widetilde{x},\xi \right) \right) \left\vert \det \frac{\partial
\theta }{\partial \xi }\left( x,\widetilde{x},\xi \right) \right\vert
\widehat{d\xi }.  \label{4.13}
\end{equation}%
From $\left( \ref{4.12}\right) $ we have, for $k=0,1,$ that $%
b_{1,\varepsilon }\left( x,\tilde{x},\theta \left( x,\widetilde{x},\xi
\right) \right) \left\vert \det \frac{\partial \theta }{\partial \xi }\left(
x,\widetilde{x},\xi \right) \right\vert $ belongs to $\Gamma _{k}^{2m}\left(
\mathbb{R}^{3n}\right) $ if $a\in \Gamma _{k}^{m}\left( \mathbb{R}%
^{2n}\right) $.

Applying the stationary phase theorem (c.f. \cite{Ro},\cite{Se} ) to $\left( %
\ref{4.13}\right) ,$ we obtain the expression of the symbol of the $h$%
-pseudodifferential operator $F_{h}F_{h}^{\ast }$:%
\begin{equation*}
\sigma (F_{h}F_{h}^{\ast })=b_{1,\varepsilon }\left( x,\tilde{x},\theta
\left( x,\widetilde{x},\xi \right) \right) \left\vert \det \frac{\partial
\theta }{\partial \xi }\left( x,\widetilde{x},\xi \right) \right\vert
_{\left\vert \widetilde{x}=x\right. }+R(x,\xi ;h)
\end{equation*}%
where $R(x,\xi ;h)\;$belongs to$\;\Gamma _{k}^{2m-2}\left( \mathbb{R}%
^{2n}\right) $ if $a\in \Gamma _{k}^{m}\left( \mathbb{R}^{2n}\right) ,$ $%
k=0,1.$

For $\tilde{x}=x,$ we have $b_{1,\varepsilon }\left( x,\tilde{x},\theta
\left( x,\widetilde{x},\xi \right) \right) =\left\vert a\left( x,\theta
\left( x,x,\xi \right) \right) \right\vert ^{2}$ where $\theta \left(
x,x,\xi \right) $ is the inverse of the mapping $\theta \rightarrow \partial
_{x}S\left( x,\theta \right) =\xi $. Thus%
\begin{equation*}
\sigma (F_{h}F_{h}^{\ast })\left( x,\partial _{x}S\left( x,\theta \right)
\right) \equiv \left\vert a\left( x,\theta \right) \right\vert
^{2}\left\vert \det \frac{\partial ^{2}S}{\partial \theta \partial x}\left(
x,\theta \right) \right\vert ^{-1}.
\end{equation*}%
such that
\begin{equation*}
\begin{gathered} \widetilde{\xi }_{\epsilon }(x,\widetilde{x},\theta (x,
\widetilde{x},\xi ))= \xi \\ \theta (x,\widetilde{x},\widetilde{\xi
}_{\epsilon }(x, \widetilde{x},\theta ))= x \\ \partial ^{\alpha }\theta
(x,\widetilde{x},\xi )=\mathcal{O} (1), \quad \forall \alpha \in
\mathbb{N}^{3n}\backslash \{0\} \end{gathered}
\end{equation*}

If we change the variable $\xi $ by $\theta (x,\widetilde{x},\xi )$ in $%
K_{1,\epsilon }(x,\widetilde{x})$, we obtain
\begin{equation*}
K_{1,\epsilon }(x,\widetilde{x})=\int_{\mathbb{R}^{n}}e^{i\langle x-\tilde{x}%
,\xi \rangle }b_{1,\epsilon }(x,\tilde{x},\theta (x,\widetilde{x},\xi ))\big|%
\det \frac{\partial \theta }{\partial \xi }(x,\widetilde{x},\xi )\big|%
\widehat{d\xi }.
\end{equation*}%
Applying the stationary phase theorem, we obtain the expression of the
symbol of the $h$-pseudodifferential operator $F_{h}F_{h}^{\ast }$, is
\begin{equation*}
\sigma (F_{h}F_{h}^{\ast })(x,\partial _{x}S(x,\theta ))\equiv |a(x,\theta
)|^{2}\big|\det \frac{\partial ^{2}S}{\partial \theta \partial x}(x,\theta )%
\big|^{-1}.
\end{equation*}
The distribution kernel of the integral operator $\mathcal{F(}F_{h}^{\ast
}F_{h})\mathcal{F}^{-1}$ is
\begin{equation*}
\widetilde{K}(\theta ,\widetilde{\theta })=\int\limits_{\mathbb{R}^{n}}e^{-%
\frac{i}{h}\left( S\left( x,\theta \right) -S\left( x,\tilde{\theta}\right)
\right) \;}\overline{a}(x,\theta )\;a\left( x,\tilde{\theta}\right) \widehat{%
dx}.
\end{equation*}%
Remark that we can deduce $\widetilde{K}(\theta ,\widetilde{\theta })$ from $%
K(x,\widetilde{x})$ by replacing $x$ by $\theta $. On the other hand, all
assumptions used here are symmetrical on $x$ and $\theta ,$ therefore $%
\mathcal{F(}F^{\ast }F)\mathcal{F}^{-1}$ is a nice $h$-pseudodifferential
operator with symbol%
\begin{equation*}
\sigma (\mathcal{F(}F_{h}^{\ast }F_{h})\mathcal{F}^{-1})\left( \theta
,-\partial _{\theta }S(x,\theta )\right) \equiv \left\vert a(x,\theta
)\right\vert ^{2}\left\vert \det \frac{\partial ^{2}S}{\partial x\partial
\theta }(x,\theta )\right\vert ^{-1}.
\end{equation*}%
Thus the symbol of $F^{\ast }F$ is given by (c.f. \cite{Ho2})
\begin{equation*}
\sigma (F_{h}^{\ast }F_{h})(\partial _{\theta }S(x,\theta ),\theta )\equiv
\left\vert a(x,\theta )\right\vert ^{2}\left\vert \det \frac{\partial ^{2}S}{%
\partial x\partial \theta }(x,\theta )\right\vert ^{-1}.
\end{equation*}
\end{proof}

\begin{corollary}
Let $F_{h}$ be the integral operator with the distribution kernel
\begin{equation*}
K(x,y;h)=\int_{\mathbb{R}^n} e^{\frac{i}{h}(S(x,\theta )-y\theta)}
a(x,\theta )\widehat{d_{h}\theta }
\end{equation*}
where $a\in \Gamma _{0}^{m}(\mathbb{R}_{x,\theta }^{2n})$ and $S$ satisfies
(G1), (G2) and (G3). Then, we have:

\begin{enumerate}
\item For any $m$ such that $m\leq 0$, $F_{h}$ can be extended as a bounded
linear mapping on $L^2(\mathbb{R}^n)$

\item For any $m$ such that $m<0$, $F_{h}$ can be extended as a compact
operator on $L^{2}(\mathbb{R}^{n})$.
\end{enumerate}
\end{corollary}

\begin{proof}
It follows from theorem \ref{main result} that $F_{h}^{\ast }F_{h}\;$is a $h$%
-pseudodifferential operator with symbol in $\Gamma _{0}^{2m}\left( \mathbb{R%
}^{2n}\right) .$

1)\ If $m\leq 0,\;$the weight $\lambda ^{2m}(x,\theta )\;$is bounded, so we
can apply the Cald\'{e}ron-Vaillancourt theorem (see \cite{CaVa,Ro,Se}) for $%
F_{h}^{\ast }F_{h}\;$and obtain the existence of a positive constant $\gamma
(n)$ and a integer $k\left( n\right) $ such that
\begin{equation*}
\left\Vert (F_{h}^{\ast }F_{h})\;u\right\Vert _{L^{2}(\mathbb{R}^{n})}\leq
\gamma (n)\;Q_{k\left( n\right) }\left( \sigma (F_{h}^{\ast }F_{h})\right)
\left\Vert u\right\Vert _{L^{2}(\mathbb{R}^{n})},\text{ }\forall u\in
\mathcal{S}(\mathbb{R}^{n})
\end{equation*}%
where%
\begin{equation*}
Q_{k\left( n\right) }\left( \sigma (F_{h}^{\ast }F_{h})\right)
=\tsum\limits_{\left\vert \alpha \right\vert +\left\vert \beta \right\vert
\leq k(n)}\sup_{(x,\theta )\in \mathbb{R}^{2n}}\left\vert \partial
_{x}^{\alpha }\partial _{\theta }^{\beta }\sigma (F_{h}^{\ast
}F_{h})(\partial _{\theta }S(x,\theta ),\theta )\right\vert
\end{equation*}%
Hence, we have $\forall u\in \mathcal{S}(\mathbb{R}^{n})$%
\begin{equation*}
\left\Vert F_{h}u\right\Vert _{L^{2}(\mathbb{R}^{n})}\leq \left\Vert
F_{h}^{\ast }F_{h}\right\Vert _{_{\mathcal{L}\left( L^{2}(\mathbb{R}%
^{n})\right) }}^{1/2}\left\Vert u\right\Vert _{L^{2}(\mathbb{R}^{n})}\leq
\left( \gamma (n)\;Q_{k\left( n\right) }\left( \sigma (F_{h}^{\ast
}F_{h})\right) \right) ^{1/2}\left\Vert u\right\Vert _{L^{2}(\mathbb{R}%
^{n})}.
\end{equation*}%
Thus $F_{h}$ is also a bounded linear operator on $L^{2}(\mathbb{R}^{n}).$

2) If $m<0,$ $\underset{\left\vert x\right\vert +\left\vert \theta
\right\vert \rightarrow +\infty }{\;\lim }\lambda ^{m}(x,\theta )=0,\;$and
the compactness theorem (see \cite{Ro,Se}) show that the operator\ $%
F_{h}^{\ast }F_{h}$ can be extended as a compact operator on $L^{2}(\mathbb{R%
}^{n}).$ Thus, the Fourier integral operator $F_{h}\;$is compact on $L^{2}(%
\mathbb{R}^{n}).$ Indeed, let $(\varphi _{j})_{j\in \mathbb{N}}$ be an
orthonormal basis of $L^{2}(\mathbb{R}^{n}),$ then
\begin{equation*}
\left\Vert F_{h}^{\ast }F_{h}-\overset{n}{\underset{j=1}{\sum }}<\varphi
_{j},.>F_{h}^{\ast }F_{h}\varphi _{j}\right\Vert \underset{n\rightarrow
+\infty }{\longrightarrow }0.
\end{equation*}%
Since $F_{h}\;$is bounded, we have $\forall \psi \in L^{2}(\mathbb{R}^{n})$%
\begin{gather*}
\left\Vert F_{h}\psi -\overset{n}{\underset{j=1}{\sum }}<\varphi _{j},\psi
>F_{h}\varphi _{j}\right\Vert ^{2}\leq  \\
\left\Vert F_{h}^{\ast }F_{h}\psi -\overset{n}{\underset{j=1}{\sum }}%
<\varphi _{j},\psi >F_{h}^{\ast }F_{h}\varphi _{j}\right\Vert \left\Vert
\psi -\overset{n}{\underset{j=1}{\sum }}<\varphi _{j},\psi >\varphi
_{j}\right\Vert
\end{gather*}%
then
\begin{equation*}
\left\Vert F_{h}-\overset{n}{\underset{j=1}{\sum }}<\varphi
_{j},.>F_{h}\varphi _{j}\right\Vert \underset{n\rightarrow +\infty }{%
\longrightarrow }0
\end{equation*}
\end{proof}

\begin{example}
We consider the function given by%
\begin{equation*}
S\left( x,\theta \right) =\sum\limits_{\substack{ \left\vert \alpha
\right\vert +\left\vert \beta \right\vert =2 \\ \alpha ,\beta \in \mathbb{N}%
^{n}}}C_{\alpha ,\beta }x^{\alpha }\theta ^{\beta },\text{ for }\left(
x,\theta \right) \in \mathbb{R}^{2n}
\end{equation*}%
where $C_{\alpha ,\beta }$ are real constants. This function satisfies $(G1),
$\textit{\ }$(G2)$\textit{\ and }$(G3).$
\end{example}

\end{document}